\documentclass[notitlepage]{amsart}

\newtheorem{pro}{Proposition}[section]
\newtheorem{teo}[pro]{Theorem}
\newtheorem{defi}[pro]{Definition}
\newtheorem{lem}[pro]{Lemma}
\newtheorem{cor}[pro]{Corollary}
\newtheorem{rk}[pro]{Remark}
\newtheorem{ex}[pro]{Example}

\newcommand{\Ext}{\mathrm{Ext}}
\newcommand{\Tor}{\mathrm{Tor}}

\newcommand{\Hom}{\mathrm{Hom}}

\newcommand{\A}{\mathcal{A}}
\newcommand{\B}{\mathcal{B}}

\newcommand{\I}{\mathcal{I}}

\newcommand{\C}{\mathcal{C}}
\newcommand{\F}{\mathcal{F}}

\newcommand{\Le}{\mathcal{L}}

\newcommand{\X}{\mathcal{X}}
\newcommand{\Y}{\mathcal{Y}}
\newcommand{\Z}{\mathcal{Z}}

\newcommand{\pd}{\mathrm{pd}}
\newcommand{\op}{\mathrm{op}}

\newcommand{\Gpd}{\mathrm{Gpd}}

\newcommand{\Proj}{\mathcal{P}}

\newcommand{\Inj}{\mathcal{I}}

\newcommand{\id}{\mathrm{id}}
\newcommand{\Gid}{\mathrm{Gid}}
\newcommand{\resdim}{\mathrm{resdim}}

\newcommand{\coresdim}{\mathrm{coresdim}}

\newcommand{\FP}{\mathrm{FP}}

\newcommand{\Modu}{\mathrm{Mod}}
\newcommand{\Ker}{\mathrm{Ker}}
\newcommand{\Ima}{\mathrm{Im}}
\newcommand{\GP}{\mathcal{GP}}

\newcommand{\GF}{\mathcal{GF}}
\newcommand{\Flat}{\mathcal{F}}

\newcommand{\glGPD}{\mathrm{gl.GPD}}
\newcommand{\glGID}{\mathrm{gl.GID}}

\newcommand{\fd}{\mathrm{fd}}
\newcommand{\Gfd}{\mathrm{Gfd}}

\newcommand{\GI}{\mathcal{GI}}
\newcommand{\Coker}{\mathrm{CoKer}}

\usepackage{latexsym,amssymb,amscd} 
\usepackage{amsmath}
\usepackage[all]{xypic}
\usepackage{amsthm}

\newcommand{\cogorro}{\vee}
\newcommand{\ortogonal}{\bot}
\newcommand{\gorro}{\wedge}

\begin{document}
\title[Gorenstein $\mathcal{FP}_n$-flat modules and weak global dimensions]{Gorenstein $\mathcal{FP}_n$-flat modules and weak global dimensions}

\author{V\'ictor Becerril}
\address[V. Becerril]{Centro de Ciencias Matem\'aticas. Universidad Nacional Aut\'onoma de M\'exico. 
 CP58089. Morelia, Michoac\'an, M\'EXICO}
\email{victorbecerril@matmor.unam.mx}
\begin{abstract} 
In this paper  we characterize the relative Gorenstein weak global  dimension of the  Gorenstein $\B$-flat $R$-modules and projectively coresolved Gorenstein $\B$-flat  $R$-modules recently studied by S. Estrada, A. Iacob, and M. A. P\'erez, which are a relativisation of the ones introduced by J. \v{S}aroch and J. \v{S}t'ov\'ich\v{e}k. As application we prove that the weak global dimension with respect to the  Gorenstein $\mathrm{FP}_n$-flat $R$-modules is finite over a Gorenstein $n$-coherent ring $R$ and in this case coincides with the flat dimension of the right $\mathrm{FP}_n$-injective $R$-modules. This result extends the known  for Gorenstein flat modules over Iwanaga-Gorenstein and Ding-Chen rings. We also show that there is a close relationship between the relative global dimension of the Gorenstein $\mathrm{FP}_n$-projectives  and the Gorenstein weak global  dimension respect to the class of Gorenstein $\mathrm{FP}_n$-flat $R$-modules.  We also get an hereditary and complete cotorsion triple and consequently a balanced pair. 
\end{abstract}  
\maketitle

\section{Introduction} 

E. E. Enochs, O. M. G.  Jenda and B. Torrecillas \cite{En95, En93} introduced Gorenstein projective, injective and flat $R$-modules for any ring $R$ and since then a Gorenstein homological algebra has been developing intensively. Most recently, there has been an increasing interest for generalizations of this families of modules. Among them are the Ding projective and Ding injective modules presented by L. X. Mao, N. Q. Ding and J. Gillespie  \cite{Gill10,Mao}. Also the named Gorenstein AC-projective and Gorenstein AC-injective modules are presented by D. Bravo, J. Gillespie and M. Hovey in \cite{BGH}. Later are introduced the Gorenstein AC-flat $R$-modules by D. Bravo, S. Estrada and  A. Iacob in \cite{BEI}.  More recently, these notions are encompassed by A. Iacob  in \cite{Alina20}, with the name of Gorenstein $\mathrm{FP}_n$-projective and Gorenstein $\mathrm{FP}_n$-injective. Under this impetus of generality in this paper we will characterize the Gorenstein weak global dimension relative to the class of the Gorenstein $ \mathcal{B}$-flat $R$-modules (where $\mathcal{B}$ a class of right $R$-modules) with application to the  corresponding class of  Gorenstein $\mathrm{FP}_n$-flat modules studied in \cite[Example 2.21 (3)]{Estrada20},  but not limited to such family \cite[Definition 2.1 (1)]{Estrada20}. We know from J. Wang and X. Zhang \cite[Lemma 1.2]{Wang23} that a ring $R$ is Gorenstein (resp. Ding-Chen) if and only if $R$ is Noetherian on both sides  (resp. coherent) with finite Gorenstein weak global dimension. Actually the Gorenstein weak global dimension has been characterized by I. Emmanouil \cite{Emmanouil} assuming finite.  Without  conditions of finiteness was proved by L. W. Christensen, S. Estrada and P. Thompson that such dimension is  symmetric  \cite{Chris-Estrada}. In what follows we prove that over $R$ a Gorenstein $n$-coherent ring \cite[\S 4]{Wang18} the left Gorenstein weak global dimension relative to the class of  Gorenstein $\mathrm{FP}_n$-flat $R$-modules, denoted $\GF_{ \mathcal{FP}_k\mbox{-}Inj (R^{\op})}$ is always finite, and for each $k \in \{0 , 1, \dots , n  \}$ coincides with the flat dimension of the class $\mathcal{FP}_k \mbox{-}Inj (R^{\op})$.  We will also see that for the Gorenstein $\mathrm{FP}_n$-flat $R$-modules the symmetry of their relative Gorenstein weak global dimensions makes sense. We also show a strong relationship between the global dimensions of the classes of projectively coresolved Gorenstein $\mathrm{FP}_n$-flat $R$-modules $\Proj\GF _{\mathcal{FP}_n \mbox{-}Inj (R^{\op})}$,   Gorenstein $\mathrm{FP}_n$-projectives  $\GP _{\mathcal{FP}_n \mbox{-}Flat (R)}$ and  Gorenstein $\mathrm{FP}_n$-injectives $\GI _{\mathcal{FP}_n \mbox{-}Inj (R)}$ by assuming conditions over the Gorenstein weak global  dimension with respect to the  Gorenstein $\mathrm{FP}_n$-flat $R$-modules $\GF _{\mathcal{FP}_n \mbox{-}Inj (R^{\op})}$, which we find quite interesting. We also give a  relativisation of the notion  of \textit{Virtually Gorenstein Ring}, by obtaining  an hereditary and complete cotorsion triple, which gives us a  balanced pair in the sense of X.-W. Chen \cite{Chen}.

\section{Preliminaries} 
In what follows, we shall work with categories of modules over an associative ring $R$ with identity. By $\Modu (R)$ and $\Modu(R^{\rm op})$ we denote the categories of left and right $R$-modules. 

\textit{Distinguished classes.} Projective, injective and flat $R$-modules will be important to present some definitions, remarks and examples. The classes of projective left and right $R$-modules will be denoted by $\mathcal{P}(R)$ and $\mathcal{P}(R^{\rm op})$, respectively. Similarly, we shall use the notations $\mathcal{I}(R)$, $\mathcal{I}(R^{\rm op})$, $\mathcal{F}(R)$ and $\mathcal{F}(R^{\rm op})$ for the classes of injective and flat modules in $\Modu(R)$ and $\Modu(R^{\rm op})$, respectively.

Concerning functors defined on $R$-modules, $\Ext ^i _R(-,-)$ denotes the right $i$-th derived functor of $\Hom_R(-,-)$. If $M \in \Modu(R^{\rm op})$ and $N \in \Modu(R)$, $M \otimes_R N$ denotes the tensor product of $M$ and $N$.

 Let $\X $ be  a class of $R$-modules and  $M \in \Modu(R)$. We set the following notation:\\
 \textit{Orthogonal classes}.  For each positive integer $i$, we consider the right orthogonal classes 
$$\X^{\ortogonal _i} := \{N \in \Modu(R): \Ext ^{i} _{R} (-,N) | _{\X} =0\} \mbox{ and } \X ^{\ortogonal} := \cap _{i >0} \X^{\ortogonal _i}.$$
Dually, we have the left orthogonal classes $^{\ortogonal _i} \X$ and $^{\ortogonal} \X$. Given a class $\Y\subseteq \Modu(R)$, we write $\X \ortogonal \Y$ whenever $\Ext ^{i} _{R} (X, Y) =0$ for all $X \in \X$,  $Y \in \Y$ and $i >0$.

 The relative projective dimension of $M \in \Modu(R)$, with respect to $\X$, is defined as
$$\pd _{\X} (M) : = \min \{n \in \mathbb{N}: \Ext ^{j} _{R} (M,-) | _{\X} =0 \mbox{ for all } j>n\}.$$
We set $\min \emptyset := \infty$. Dually, we denote by $\id _{\X} (M)$ the relative injective dimension of $M$ with respect to $\X$. Furthermore,  we set  
$$\pd _{\X} (\Y) := \sup\{ \pd _{\X} (Y): Y \in \Y\} \mbox{ and } \id _{\X} (\Y) := \sup \{\id _{\X} (Y): Y \in \Y\}. $$
If $\X = \Modu(R)$, we just write $\pd (\Y)$ and $\id (\Y)$. Similarly with the $i$-th derived functor of $-\otimes _R -$, the notation $\fd (\X)$, means the flat dimension of the class $\X$, this is $\sup \{\fd (X) : X \in \X \}$.

 \textit{Resolution and coresolution dimension}. The \textbf{$\X$-coresolution dimension of $M $}, denoted  by $\coresdim _{\X} (M)$, is the smallest non-negative integer $n$ such that there is an exact sequence 
$$0 \to M \to X_0 \to X_1 \to \cdots \to X_n \to 0,$$ 
with $X_i \in \X$ for all $i \in \{0, \dots , n\}$. If such $n$ does not exist, we set $\coresdim _{\X} (M) := \infty$. Also, we denote by $\X^{\cogorro } _n$ the class of $R$-modules with $\X$-coresolution dimension at most $n$. The union $\X ^{\cogorro} := \cup _{n \geq 0} \X ^{\cogorro} _n$ is the class of $R$-modules with finite $\X$-coresolution dimension.
Dually, we have the $\X$-{\bf{resolution dimension}}  of $M$, denoted by $\resdim_\X\,(M)$,   $\mathcal{X}^{\wedge} _n$  the class $R$-modules having  $\mathcal{X}$-resolution dimension at most $n$, and the union $\X ^{\gorro} := \cup _{n \geq 0} \X ^{\gorro} _n$ is the class of $R$-modules with finite $\X$-resolution dimension. We set  
$$\coresdim_\X\,(\Y):=\mathrm{sup}\,\{\coresdim_\X\,(Y)\;
:\; Y\in\Y\},$$ and  $\resdim_\X\,(\Y)$ is defined dually.

\textit{Relative Gorenstein $R$-modules.} Recall that given a class $\mathcal{B} \subseteq \Modu (R^{\op})$ a left $R$-module $G$ is said to be
\begin{itemize}
 \item[(a)] \textbf{Gorenstein $\mathcal{B}$-flat} \cite[Definition 2.1]{Estrada20}, if it is isomorphic to $\Ker (F^0 \to F^1) $ of an exact complex of  flat left $R$-modules 
$$ \cdots \to F_1 \to F_0 \to F^0 \to F^1 \to \cdots$$
which remains acyclic whenever that the functor $B \otimes _R -$ is applied to it for any $B \in \mathcal{B}$. The class of all Gorenstein $\mathcal{B}$-flat $R$-modules is denoted by $\GF _{\mathcal{B}} (R)$. 
\item[(b)]  \textbf{Projectively coresolved Gorenstein $\mathcal{B}$-flat}  \cite[Definition 2.6]{Estrada20}, if it is isomorphic to $\Ker (P^0 \to P^1) $ of an exact complex of projective left $R$-modules 
$$ \cdots \to P_1 \to P_0 \to P^0 \to P^1 \to \cdots$$
which remains acyclic whenever that the functor $B \otimes _R -$ is applied to it for any $B \in \mathcal{B}$. The class of all projectively coresolved Gorenstein $\mathcal{B}$-flat modules is denoted by $\Proj\GF _{\mathcal{B}} (R)$. 
\end{itemize}

For any $M \in \Modu (R)$ we define the  \textbf{Gorenstein $\mathcal{B}$-flat dimension of $M$}, denoted  $\Gfd _{\mathcal{B}} (M)$ as $\resdim_{\GF _{\mathcal{B}} } (M)$. Similarly we declare the \textbf{projectively coresolved Gorenstein $\mathcal{B}$-flat dimension of $M$}, denoted  $P\Gfd _{\mathcal{B}} (M)$ as $\resdim _{\Proj\GF _{\mathcal{B}}} (M)$. Monomorphism and epimorphism in $\Modu(R)$ may sometimes be denoted using arrows $\hookrightarrow$ and $\twoheadrightarrow$, respectively.

\textit{Cotorsion pairs.} A pair $(\X, \Y) \subseteq \Modu (R)^{2} $ is a \textbf{cotorsion pair} if $\X^{\ortogonal _1} = \Y$ and $\X = {^{\ortogonal _1} \Y}$. This cotorsion pair is  \textbf{complete} if for any $C \in \Modu(R)$, there are exact sequences $0 \to Y \to X \to C \to 0$ and $0 \to C \to Y' \to X' \to 0$ where $X, X' \in \X $ and $Y, Y' \in \Y$. 

Moreover, the cotorsion pair is \textbf{hereditary} if $\X \ortogonal \Y$. A triple $(\X,\Y,\Z) \subseteq \Modu (R) ^{3}$ is called a \textbf{cotorsion triple} \cite{Chen} provided that both $(\X, \Y)$ and $(\Y, \Z)$ are cotorsion pairs; it is \textbf{complete} (resp. \textbf{hereditary}) provided that both of the two cotorsion  pairs are complete (resp. hereditary).

\section{ Gorenstein weak  global dimensions}

We begin consider the class of projective coresolved Gorenstein flat modules generalized, proposed by S. Estrada, A. Iacob and M. A. P\'erez \cite[Definition 2.6]{Estrada20}, which are a generalization of the \v{S}aroch and \v{S}t'ov\'ich\v{e}k's \cite[Section 3]{Stovi}.  Based on the fact that they are a type of relative Gorenstein flat, we will study their relative weak global dimension. We  start with the following result.

\begin{pro}\label{Estable}
Let $\mathcal{B} \subseteq \Modu (R ^{\op})$ be a class of right $R$-modules. For all $M \in \Proj \GF  ^{\gorro} _{\mathcal{B}}$  the following conditions are equivalent:
\begin{itemize}
\item[(i)] $P\Gfd _{\mathcal{B}}(M) \leq n$,
\item[(ii)] If $0 \to K_n \to H_{n-1} \to \cdots \to H_0 \to M \to 0$ is an exact sequence with $ H_i \in \Proj \GF  _{\mathcal{B}}$, then $ K_n  \in \Proj \GF  _{\mathcal{B}}$.
\end{itemize}
\end{pro}
\begin{proof}
Is an formal consequence of their properties \cite[Theorem 2.10]{Estrada20}, the Eilenberg's Swindle and \cite[Lemma 3.12]{Au-Bri}.
\end{proof}
We have the following immediate result. 
  
  \begin{cor} \label{CoroDimensions}
 Let $\mathcal{B} \subseteq \Modu (R ^{\op})$ be a class of right $R$-modules and consider an exact sequence $0 \to K \to G \to M \to 0$ such that  $G \in \Proj \GF _{\mathcal{B}}$. If $M \in \Proj \GF  _{\mathcal{B}}$, then also $K \in  \Proj \GF  _{\mathcal{B}}$. Otherwise we get $P \Gfd  _{\mathcal{B}} (K ) = P \Gfd  _{\mathcal{B}} (M) -1 \geq 0$.
  \end{cor}
  
For the following we need to keep in mind that  the class $\Proj \GF _{\mathcal{B}}$ is always  a  resolving class \cite[Theorem 2.10]{Estrada20}.
  
  \begin{pro} \label{Holm2.19}
Let $\mathcal{B} \subseteq \Modu (R ^{\op})$ be a class of right $R$-modules and $\{ M_{\lambda}\} _{\lambda \in \Lambda}$ a family of left $R$-modules, then the following equality is true
  $$P \Gfd  _{\mathcal{B}} (\bigoplus _{\Lambda} M_{\lambda}) = \sup \{ P \Gfd  _{\mathcal{B}} (M_{\lambda}) |  \lambda \in \Lambda \}.$$
  \end{pro}
  
  \begin{proof}
  The class $\Proj \GF _{\mathcal{B}}$  is closed by coproducts, thus we get the inequality $P \Gfd  _{\mathcal{B}} (\bigoplus _{\Lambda} M_{\lambda}) \leq \sup \{ P \Gfd  _{\mathcal{B}} (M_{\lambda}) |  \lambda \in \Lambda \}$. For the converse inequality is enough to show that if $N$ is a direct summand of an $R$-module $M$, then $P \Gfd  _{\mathcal{B}} (N) \leq P \Gfd  _{\mathcal{B}} (M)$. To do this we can suppose that $P \Gfd  _{\mathcal{B}} (M) =n < \infty$ and thus proceed by induction over $n$.  When $n = 0$, from \cite[Theorem 2.10]{Estrada20} we know that  $\Proj \GF _{\mathcal{B}}$ is closed by kernels of epimorphisms and clearly is closed by coproducts. Thus, by Eilenberg's Swindle is closed by summands, with this we get that if $M$ lies in   $\Proj \GF _{\mathcal{B}}$ so is $N$. If $n >0$, we have $M = N \oplus L$ and we can pick exact sequences $0 \to K_0 \to G_0 \to N \to 0$ and $0  \to K _1 \to G_1\to L \to 0$ where $G_0 , G_1 \in \Proj \GF _{\mathcal{B}}$. We can construct the following commutative diagram with split-exact rows 
  $$\xymatrix{ 
  K_{0_{}}  \ar@{^{(}->}[d]  \ar@{^{(}->}[r] & K_0 \oplus K_1 \ar@{^{(}->}[d]  \ar@{>>}[r] & K_{1_{}}   \ar@{^{(}->}[d]  \\
  G_0  \ar@{>>}[d] \ar@{^{(}->}[r] & G_0 \oplus G_1 \ar@{>>}[r]\ar@{>>}[d] & G_1  \ar@{>>}[d] \\
 N \ar@{^{(}->} [r]& M \ar@{>>}[r] & L
}$$

Applying Corollary \ref{CoroDimensions} to the middle column in this diagram we obtain that $P \Gfd  _{\mathcal{B}} (K_0 \oplus K_1) = n-1$. Hence the induction hypothesis give us $P \Gfd  _{\mathcal{B}} (K_0) \leq n-1$, and thus the short exact sequence $0 \to K_0 \to G_0 \to N \to 0$ shows that $P \Gfd  _{\mathcal{B}} (N) \leq n$, as desired.
  \end{proof}
  
  We are now in a position to prove the first main result in this paper.   For this, given a class of right $R$-modules $\mathcal{B} \subseteq \Modu (R^{\op})$ the  \textit{left  Gorenstein weak  global dimension}  relative to the class $\Proj\GF _{ \mathcal{B}  }$ of projectively coresolved Gorenstein $ \mathcal{B} $-flat $R$-modules is
   $$\mathrm{l.w.PGFgl} _{\mathcal{B} }(R):= \sup \{P\Gfd _{ \mathcal{B} } (M)| M \in \Modu (R)  \}.$$
   Also the notation $\Proj \GF ^{2} _{\mathcal{B}} $ is used to refer to the class of left $R$-modules $M$ for which there exists an exact sequence of left $R$-modules in $\Proj \GF _{\mathcal{B}} $
   $$\cdots \to G_1 \to G_0 \to G^0 \to G^1 \to \cdots $$
   such that $M \cong \Ker (G^0 \to G^1) $ and such that remains acyclic  when the functor $B  \otimes_R - $  is applied to it for all $B \in  \mathcal{B}$.

\begin{pro}\label{WeakCoresolved}
Consider  a class $\mathcal{B} \subseteq \Modu (R^{\op})$ of right $R$-modules and let  $n$ be a non-negative integer. If $\fd (\mathcal{B}) < \infty$ and  $P\Gfd _{\mathcal{B}} (\Inj (R)) \leq n$ then $\mathrm{l.w.PGFgl} _{\mathcal{B} }(R) \leq n$.
\end{pro}

\begin{proof}
Take $M \in \Modu (R)$, and consider a projective and a injective resolution to $M$, 
$$ \cdots \to P_1 \to P_0 \to M \to 0, \;\;\mbox{ and } \;\;  0 \to M \to I_0 \to I_1 \to \cdots ,$$
for $P_{-1} = M = I_{-1}$ and  $i \geq 0 $ we can declare $C_i := \Ima (P_i \to P_{i-1} )$ and $K_i := \Ker (I_i \to I_{i+1})$ to obtain a decomposition of such resolutions in short exact sequences for each $i \geq 0$ as follows 
$$ 0 \to C_{i+1} \to P_i \to C_i \to 0  \;\;\mbox{ and } \;\; 0 \to K_i \to I_i \to K_{i+1} \to 0.$$
By doing the corpoduct of the first family and the product of the second family we get the exact sequences 
\[  0 \to \bigoplus _{i \in \mathbb{N}} C_{i+1} \to \bigoplus _{i \in \mathbb{N}} P_i \to  M \oplus (\bigoplus _{i \in \mathbb{N}} C_{i+1}) \to 0, \]

\[0 \to M \oplus (\prod _{i \in \mathbb{N}} K_{i+1} )\to \prod _{i \in \mathbb{N}} I_i \to   \prod _{i \in \mathbb{N}} K_{i+1} \to 0,
\]
and adding both of them we get the exact sequence 
$$0 \to M \oplus (C \oplus K) \to P \oplus I \to M \oplus (C \oplus K) \to 0,$$
with $C := \bigoplus _{i \in \mathbb{N}} C_{i+1}$, $K:= \prod _{i \in \mathbb{N}} K_{i+1} $, $P := \bigoplus _{i \in \mathbb{N}}  P_i$ and $I:= \prod _{i \in \mathbb{N}} I_i$. Take an exact sequence  $\cdots \to F_1 \to F_0 \to M\oplus (C \oplus K)$ by projectives $F_i$. Using Horseshoe's Lemma we obtain the following  commutative and exact diagram
$$\xymatrix{ L_{n_{}\;} \ar@{^{(}->}[d]  \ar@{^{(}->}[r]& F_{} \ar@{^{(}->}[d]  \ar@{>>}[r]& L_{n_{}\;} \ar@{^{(}->}[d]  \\
 F_{n-1} \ar[d]  \ar@{^{(}->}[r]& F_{n-1} \oplus F_{n-1} \ar[d]  \ar@{>>}[r]& F_{n-1} \ar[d]  \\
  \vdots \ar[d]  & \vdots \ar[d] & \vdots  \ar[d]  \\
  F_0 \ar@{>>}[d] \ar@{^{(}->}[r] & F_0 \oplus F_0 \ar@{>>}[r]\ar@{>>}[d] & F_0  \ar@{>>}[d] \\
 M\oplus (C \oplus K) \ar@{^{(}->} [r]& P \oplus I \ar@{>>}[r] & M\oplus (C \oplus K), 
}$$
note that $P\Gfd _{\mathcal{B}} (P \oplus I) = P\Gfd _{\mathcal{B}} (I)$ by  Proposition  \ref{Holm2.19}, since $P \in \Proj\GF_{\mathcal{B}}$. Now using the hypothesis we have $P\Gfd _{\mathcal{B}} (I) \leq n$. Therefore by  Lemma \ref{Estable}  $F \in \Proj \GF_{ \mathcal{B}} $ since $F_i \in \Proj (R) \subseteq \Proj\GF _{ \mathcal{B}} $. 

Now take $B \in \mathcal{B}$, by hypotheses $\fd (B) \leq n$, then using the exact sequence by projectives for  $M \oplus (C \oplus K)$ we have  $\Tor _1 ^R (B , L_n) \cong \Tor _{n+1} ^R (B , M \oplus (C \oplus K)) =0$. In consequence the exact sequence $0 \to L_n \to F \to L_n \to 0$ is $(\mathcal{B} \otimes _R -)$-acyclic. Pasting this sequence repeatedly we obtain an exact sequence $$ \cdots \to F \to F \to F \to \cdots $$ which is $(\mathcal{B} \otimes _R -)$-acyclic. This implies that  $L_n \in \Proj \GF^2 _{\mathcal{B}}$. 
From  \cite[Theorem 2.7]{ZWang} we know that $\Proj \GF ^2 _{\mathcal{B}} = \Proj \GF _{\mathcal{B}}$ since $\Proj \GF _{\mathcal{B}}$ is closed by extensions, therefore $L_n \in  \Proj \GF _{\mathcal{B}} $ and from this we conclude by Proposition \ref{Holm2.19} that $P\Gfd _{\mathcal{B}} (M ) \leq P\Gfd _{\mathcal{B}} (M \oplus (C \oplus K)) \leq n$ as desired.
\end{proof}

Now we give a converse of the previous result. 

\begin{pro}
 Let $\mathcal{B} \subseteq \Modu (R^{\op})$ be a class of right $R$-modules and $M \in \Modu (R)$. Consider the following statements:
 \begin{itemize}
 \item[(i)] $P\Gfd _{\mathcal{B}} (M) \leq n $,
\item[(ii)]  $P\Gfd  _{\mathcal{B}} (M) < \infty$ and $\Tor _{i} ^{R} (B, M) =0$ for all $i > n$ and all $B \in \mathcal{B} $,
\item[(iii)] $P\Gfd  _{\mathcal{B}} (M) < \infty$ and $\Tor _{i} ^{R} (E, M) =0$ for all $i > n$ and all $E \in \mathcal{B} ^{\cogorro} $.
 \end{itemize}  
 Then $\mathrm{(i)} \Rightarrow \mathrm{(ii)}  \Rightarrow  \mathrm{(iii)}$. In consequence $\mathrm{l.w.PGFgl} _{\mathcal{B} }(R) < \infty$ if and only if, both $\fd (\mathcal{B}) $ and  $P\Gfd  _{\mathcal{B}} (\Inj (R))$ are finite.
 
\end{pro}
\begin{proof}
(i) $\Rightarrow$ (ii). By induction over $n$. The case $n =0$ is true by \cite[Lemma 2.8]{Estrada20}. Now we can assume that $n >0$. Thus there is an exact sequence $0 \to K \to G \to M \to 0$, with $G \in \Proj\GF _{ \mathcal{B}} $ and $P\Gfd  _{\mathcal{B}} (K) =  n-1$. We know that for all $B \in \mathcal{B}$ it is satisfied that $\Tor _i ^R (B, G) =0$  for all $i > 0$ and that $\Tor (B, K) =0$ for all $i > n-1$ (by induction). Then we use the long exact sequence $\Tor _{i+1} ^R (B,G) \to \Tor _{i+1} ^R (B, M) \to \Tor _i ^R (B, K)$ to conclude that $\Tor _{i+1} ^R (B, M ) =0$ for all $i > n-1$.

(ii) $\Rightarrow$ (iii). It follows by a shifting argument. 

For the last statement suppose  that $\mathrm{l.w.PGFgl} _{\mathcal{B} }(R)  =n < \infty$, then by (i) $\Rightarrow$ (ii) we have that $\fd (\mathcal{B}) \leq n < \infty$ and also for the subclass $\Inj (R) \subseteq \Modu (R)$ we have $P\Gfd  _{\mathcal{B}} (\Inj (R)) \leq \mathrm{l.w.PGFgl} _{\mathcal{B} }(R) < \infty$. Finally the other implication is given by Proposition \ref{WeakCoresolved}.
\end{proof}

Now we turn our attention to the class   $\GF _{ \mathcal{B}  }$ of  Gorenstein $ \mathcal{B} $-flat $R$-modules. We need before the following tool, where  $(\mathcal{B} \otimes _R - )$-acyclic means that $(B\otimes _R -)$-acyclic for all $B \in \mathcal{B}$
\begin{lem}\label{tensor-exacto}
Consider $\mathcal{B} \subseteq \Modu (R ^{\op})$ a class of right $R$-modules such that  $\fd (\mathcal{B}) < \infty$. Then, all exact complex of  flat $R$-modules  is $(\mathcal{B} \otimes _R - )$-acyclic.
\end{lem} 

\begin{proof}
Assume that  $m := \fd (\mathcal{B}) < \infty$ and take $ F ^{\bullet} : \cdots \to F_1 \to F_0 \to F^0 \to F^1 \to \cdots $ an exact complex  of flat $R$-modules. Since $\fd (\mathcal{B}) = m$, for  $B \in \mathcal{B} $  there is an exact sequence $0 \to P_m \to \cdots \to P_1 \to P_0 \to B \to 0$ with all $P_{i} \in \Flat (R)$. For each  $j \in \{ 0,1, \dots , m-1\} $ consider the short exact sequences $0 \to K_j \to P_j \to K_{j-1} \to 0$ where $K_{m-1} := P_{m}$ and $K_{-1} := B$. For each $j$ we have an exact sequence of complexes $0 \to K_j \otimes _R F^{\bullet} \to P_j \otimes _R F^{\bullet} \to K_{j-1} \otimes _R F^{\bullet} \to 0$. Since $K_{m-1} = P_{m}$ we get that $K_{m-2} \otimes _R F^{\bullet}$ is acyclic, so continuing iteratively we obtain that $B \otimes _R F^{\bullet}$ is acyclic.
\end{proof}

We have now the following which extends the known in  \cite[Lemma 5.2 and Theorem 5.3]{Emmanouil}. 
Recall the notion of \textit{left  Gorenstein weak global dimension} relative to the class $\GF _{ \mathcal{B}  }$ of Gorenstein $ \mathcal{B} $-flat $R$-modules as follows
$$\mathrm{l.w.Ggl} _{\mathcal{B} }(R):= \sup \{\Gfd _{ \mathcal{B} } (M)| M \in \Modu (R)  \} .$$

\begin{pro} \label{Ikenaga}
Consider $\mathcal{B} \subseteq \Modu (R^{\op})$ a class of right $R$-modules such that  $\fd (\mathcal{B} ) < \infty$. Then $\fd (\Inj (R)) \leq n$ implies that $\mathrm{l.w.Ggl} _{\mathcal{B}} (R) \leq n$.
\end{pro}

\begin{proof}
Take $M \in \Modu (R)$, we can consider an exact sequence  $0 \to M \to I^{0} \to I^1 \to \cdots$, with $I ^{j } \in \Inj (R)$ for all $j \geq 0$. From the inequality  $\Gfd  _{\mathcal{B} }(\Inj (R))\leq n$ and \cite[Ch. XVII, \S 1 Proposition 1.3]{Cartan} we can construct the following commutative diagram
$$\xymatrix{  
 0\ar[r] & Q_{_{}} \ar@{^{(}->}[d] \ar@{^{(}->}[r] &  P^0 _{n_{}}  \ar@{^{(}->}[d] \ar[r]& P^1 _{n_{}}  \ar@{^{(}->}[d] \ar[r] & \cdots \\
0 \ar[r]& Q_{n-1} \ar[r] \ar[d] & P^0 _{n-1} \ar[d]  \ar[r]&  P^1 _{n-1} \ar[d]  \ar[r] & \cdots  \\
    &  \vdots  \ar[d] & \vdots  \ar[d] & \vdots \ar[d] & \\
 0\ar[r] & Q_0 \ar@{>>}[d] \ar[r] & P^0 _0  \ar@{>>}[d] \ar[r]& P^1 _0  \ar@{>>}[d] \ar[r] & \cdots \\
0 \ar[r]& M \ar@{^{(}->}[r]  & I^0   \ar[r] &  I^1 \ar[r]  & \cdots  }$$
such that  $P_i ^j \in \Proj (R)$ for all $i \in \{0,1, \dots , n-1 \}$ and all $j\geq 0$, and where $P^j _n \in \Flat (R)$ for all $j \geq 0$. With  $Q _i := \Ker (P^0 _i \to P^1 _i)\in \Proj (R)$ for all $i \in \{0 , \dots , n-1\}$ and 
\begin{center}
$\lambda  : 0 \to Q \to P_n ^0 \to P_n ^1 \to \cdots$\\
 $\theta : 0 \to Q \to Q_{n-1} \to \cdots \to Q_0 \to M \to 0$
 \end{center}
  exact complexes. We can complete $\lambda$  to an exact complex of  modules in $\Flat (R)  $ by consider an projective resolution of $Q$, 
  $$ \cdots \to P_1 \to P_0 \to Q \to 0.$$ 
  Since $\fd (\mathcal{B} ) < \infty$ we know from Lemma \ref{tensor-exacto} that the exact complex of flat $R$-modules  resulting $$\cdots \to P_1 \to P_0 \to P_n ^0 \to P_n ^1 \to \cdots  ,$$
   is $(\mathcal{B} \otimes _R -)$-acyclic, therefore $Q \in  \GF  _{\mathcal{B}} $. Thus from the exact sequence $\theta$ we conclude that $\Gfd _{\mathcal{B} } (M) \leq n $. 
\end{proof}
The special case when $\mathcal{B} : = \Inj (R^{\op})$  has been studied by Emmanouil  \cite[Theorem 5.3]{Emmanouil}, where the condition $\mathrm{l.w.Ggl} _{\Inj (R) }(R) < \infty$  is assumed. Note that in our previous result we drop this finiteness condition. Recently this has also been noted for this special case by L. W.  Christensen, S. Estrada and P. Thompson in  \cite[Theorem 2.4]{Chris-Estrada} where also is studied the symmetric properties of the Gorenstein weak global dimensions. We will study this notion later for a suitable family of $R$-modules.

From  the results in \cite{Becerril22}, we have additionally a characterization of the left   Gorenstein weak global dimension relative to the class $\GF _{\mathcal{B}}$.

\begin{cor} \label{caracteriza}
Let  $\mathcal{B} \subseteq \Modu (R^{\op})$ be a class of right $R$-modules with  $\Inj (R^{\op}) \subseteq \mathcal{B}$ and such that  $\GF _{\mathcal{B}}$ is closed by extensions. Then the quantity $\mathrm{l.w.Ggl} _{\mathcal{B}} (R) $ is finite if and only if both  $\fd (\mathcal{B}) $ and $\fd  (\Inj (R))$ are finite.
\end{cor}

\begin{proof}
From the previous result we only need to prove $\Rightarrow $). For do this, from \cite[Proposition 4.1]{Becerril22} we have that $\fd (\mathcal{B}) < \infty$. Since $\Inj (R^{\op})  \subseteq \mathcal{B}$, then the containment  $\GF _{\mathcal{B}} \subseteq \GF _{\Inj (R^{\op})}$ is given. Thus  for all $M \in \Modu (R)$ their Gorenstein flat dimensions are such that $$\Gfd _{\Inj (R^{\op})} (M) \leq \Gfd _{\mathcal{B}} (M),$$ thus the finiteness of $\mathrm{l.w.Ggl} _{\mathcal{B}} (R) $ implies the finiteness of the classical  Gorenstein weak global dimension  $\mathrm{l.w.Ggl} _{\Inj (R^{\op})} (R) $. From  \cite[Theorem 2.4]{Chris-Estrada} we obtain  that $\fd (\Inj (R)) < \infty$.
\end{proof}

\subsection{Applications}
In a stimulating talk given by Daniel Bravo for the University of the Republic of Uruguay, the speaker mentioned several of the known properties related to the class of $R$-modules of type FP$_n$, as well as possible directions in which their study could go. It is partly for this reason that we have been interested in studying such class of $R$-modules and contributing with  some  results. In what follows we study the weak global dimension concerning the class of Gorenstein $\mathcal{FP}_n$-flat $R$-modules coming  from the class of right $\mathcal{FP}_n$-injectives. 
\begin{ex}
 We address here the main example of our theory.
\begin{enumerate}
\item An $R$-module $F$ is of \textbf{type FP$_{\infty}$}  \cite{BGH} if it has a projective resolution $\cdots \to P_1 \to P_0 \to F \to 0$, where each $P_i$ is finitely generated.  An $R$-module $A$ is \textbf{absolutely clean} if $\Ext ^{1} _{R} (F,A) =0$ for all $R$-modules $F$ of type FP$_{\infty}$. Also an $R$-module $L$ is \textbf{level} if $\mathrm{Tor} _{1} ^{R} (F,L) =0$ for all (right) $R$-module $F$ of type FP$_{\infty}$. We denote by $\mathrm{AC} (R)$ (also denoted $\mathcal{FP}_{\infty} \mbox{-}Inj (R)$  \cite[Definition 3.1]{BP}) the class  of all absolutely clean $R$-modules, and by  $\mathrm{Lev} (R)$ (also denoted $\mathcal{FP}_{\infty} \mbox{-}Flat (R)$ \cite[Definition 3.2]{BP})  the class of all level $R$-modules. 
 The class $\GF _{(\Flat (R) , \mathrm{AC}(R ^{\op}))}$ is called \textbf{Gorenstein AC-flat} \cite{BEI}.
\item $M \in \Modu (R)$ is called \textbf{finitely $n$-presented} if $M$ possesses a projective resolution $0 \to P_n \to \cdots \to P_1 \to P_0 \to M \to 0$ for which each $P_i$ is finitely generated free. We denote by $\mathcal{FP}_n(R)$ the class of all finite $n$-presented $R$-modules.  An $R$-module $E$ is called \textbf{$\mathcal{FP}_n$-injective} if for all $M \in \mathcal{FP}_n(R)$, $\Ext _R ^1 (M, E) =0$. Also an $R^{\op}$-module $H$ is called \textbf{$\mathcal{FP}_n$-flat} if for all $M \in \mathcal{FP}_n(R)$, $\Tor _1 ^R (H, M) =0$.
These classes are denoted by   $\mathcal{FP}_n\mbox{-}Inj (R )$  and $ \mathcal{FP}_n\mbox{-}Flat (R^{\op}) $, respectively.
\item  There is a tern of Gorenstein modules that comes from the classes  $ \mathcal{FP}_n\mbox{-}Flat (R) $ and $\mathcal{FP}_n\mbox{-}Inj (R)$, such classes of   Gorenstein $R$-modules  have been studied recently in \cite{Alina20} where they are called as  \textbf{Gorenstein $\mathcal{FP}_n$-projective}, \textbf{Gorenstein $\mathcal{FP}_n$-injective} and \textbf{Gorenstein $\mathcal{FP}_n$-flat} and here respectively denoted  by $\GP_{\mathcal{FP}_n\mbox{-}Flat (R) }$, $\GI _{\mathcal{FP}_n\mbox{-}Inj(R) }$, $\GF _{\mathcal{FP}_n\mbox{-}Inj(R^{\op})}$.   These classes were also studied in \cite{Estrada20}, where also is considered the class of \textbf{projectively coresolved Gorenstein $\mathcal{FP}_n$-flat} $R$-modules, denoted $\Proj\GF _{\mathcal{FP}_n\mbox{-}Inj(R^{\op})}$.
\end{enumerate}
\end{ex}

\begin{rk} \label{Remark1}
\begin{enumerate}
 \item We know by \cite[Proposition 3.10]{BP} and from \cite[Example 2.21 (3)]{Estrada20}, that the class $\GF _{ \mathcal{FP}_n\mbox{-}Inj (R^{\op})}$ is closed by extensions.
 \item Note that the pairs $(\Proj (R),\mathcal{FP}_n\mbox{-}Flat (R) )$ and $(\mathcal{FP}_n\mbox{-}Inj (R), \Inj (R))$ are GP-admisible and GI-admisible respectively \cite[Definitions 3.1 \& 3.6]{BMS}.
 \end{enumerate}
\end{rk}

It is clear that always there is a containment $\mathcal{FP}_n(R) \supseteq \mathcal{FP}_{n+1}(R)$, i.e. there is an descending chain 
$$\mathcal{FP}_0(R) \supseteq \mathcal{FP}_1(R) \supseteq \cdots \supseteq \mathcal{FP}_{n+1}(R)  \supseteq \cdots \supseteq \mathcal{FP}_{\infty}(R),$$ 
can bee seen also that $R$ is left coherent, if and only if, the containment  $\mathcal{FP}_1(R) \subseteq \mathcal{FP}_{2}(R)$ is given \cite[Proposition 2.1]{BP}. Also know that $R$ is left Noetherian if and only if every module finitely generated is finitely presented, i.e. $\mathcal{FP}_0(R) \subseteq \mathcal{FP}_1(R)$.  From this we have the following notion given by D. Bravo and M. A. P\'erez.

\begin{defi}\cite[Definition 2.2]{BP} \label{G-coherent}
A ring $R$ is left $n$-coherent if every finitely $n$-presented left $R$-module is finitely $(n+1)$-presented, that is $\mathcal{FP}_n(R) \subseteq \mathcal{FP}_{n+1}(R)$.
\end{defi}

J. P. Wang, Z. K. Liu, X.Y. Yang \cite{Wang18} introduced the notion of \textit{Goresntein $n$-coherent ring}. It is defined as a two sided  $n$-coherent ring $R$ for which  there is some  $m \in \mathbb{Z} _{\geq 0}$ such that $\fd (M), \fd (N) \leq m$ for all $M \in \mathcal{FP}_n\mbox{-}Inj (R)$ and $N \in \mathcal{FP}_n\mbox{-}Inj (R^{\op})$. In this context  we have that $\fd (\Inj (R)) < \infty$ since $\Inj (R) \subseteq \mathcal{FP}_n\mbox{-}Inj (R)$. This  implies that $\Gfd _{\mathcal{FP}_n\mbox{-}Inj (R^{\op})} (\Inj (R)) < \infty$. Therefore, from Corollary \ref{caracteriza} we conclude that the dimension  $\mathrm{l.w.Ggl}_{\mathcal{FP}_n\mbox{-}Inj (R^{\op})}(R)$ is finite. We declare this fact  as follows.

\begin{pro}
Let  $R$ be a Gorenstein n-coherent ring and $\mathcal{FP}_n\mbox{-}Inj (R^{\op})$ the class of right $\mathcal{FP}_n$-injective $R$-modules. Then the left  Gorenstein weak global dimension relative to the class $\GF_{\mathcal{FP}_n\mbox{-}Inj (R^{\op})}$
$$\mathrm{l.w.Ggl}_{ \mathcal{FP}_n\mbox{-}Inj (R^{\op})}(R)$$  is finite and coincides with $\fd (\mathcal{FP}_n\mbox{-}Inj (R^{\op}))$.
\end{pro} \label{Interaccion}

Note that in the sense of Definition \ref{G-coherent} a Gorenstein ring is precisely a Gorenstein $0$-coherent  ring, and a Gorenstein $1$-coherent ring is precisely a Ding-Chen ring \cite{DingChen}. Thus, the previous result  recovers   the known for these rings \cite[Lemma 1.2]{Wang23}, and  extends it to its relative version. 

We will now see how the global dimensions of different kinds of relative Gorenstein $R$-modules interact with each other, for this we recall some notation adapted to our context. 
\begin{defi}\cite[Definition 4.17]{BMS} Let $\X,\Y$ be a pair of classes of left $R$-modules. Consider the following  homological dimensions. 
\begin{itemize}
\item[(1)] The {\bf global  $\X$-Gorenstein projective dimension} of $R$ 
\begin{center} $\glGPD_{\X}(R):=\sup \{ \Gpd_{\X} (M)|  M \in \Modu (R) \}.$\end{center}
\item[(2)] The {\bf global  $\Y$-Gorenstein injective dimension}  of $R$  
\begin{center}$\glGID_{\Y}(R):= \sup \{ \Gid_{\Y}(M) | M \in \Modu (R) \}. $\end{center}
\end{itemize}
\end{defi}
In our context we will work with the classes $ \X = \mathcal{FP}_n\mbox{-}Flat (R)$ and $\Y = \mathcal{FP}_n\mbox{-}Inj (R)$ as follows.
\begin{pro} \label{Globales-iguales}
Let $R$ be a ring such that $\fd (\mathcal{FP}_n\mbox{-}Inj (R^{\op})) < \infty$ and $\pd (\Inj (R)) < \infty$. Then, the following statements are  true:
\begin{itemize}
\item[(i)] $\mathrm{l.w.PGFgl}_{ \mathcal{FP}_n\mbox{-}Inj (R^{\op})}(R)$ is finite,
\item[(ii)]  $\id (\mathcal{FP}_n\mbox{-}Flat (R))$ and $\mathrm{gl.GPD} _{\mathcal{FP}_n\mbox{-}Flat (R)} (R)$ are finite for $n> 1$,
\item[(iii)] $\pd (\mathcal{FP}_n\mbox{-} Inj (R)) $ and $\mathrm{gl.GID}_{\mathcal{FP}_n\mbox{-}Inj (R)} (R)$ are finite, for $n> 1$.
\end{itemize}
\end{pro}

\begin{proof}
(i) Since $\pd (\Inj (R)) < \infty$ and $\Proj (R) \subseteq \Proj \GF_{\mathcal{FP}_n\mbox{-}Inj (R^{\op})}$ it follows that $P\Gfd (\Inj (R)) \leq \pd (\Inj (R))$. Therefore the result follows from  Proposition \ref{WeakCoresolved}.

(ii) We know from \cite[Theorem 3.6]{Alina20} that $\Proj \GF _{\mathcal{FP}_n\mbox{-}Inj (R^{\op})} = \GP _{\mathcal{FP}_n\mbox{-}Flat (R)}$ for $n>1$. From (i), Remark \ref{Remark1} (2) and  \cite[Proposition 3.7 (i)]{Becerril22-1} we get $\id (\mathcal{FP}_n\mbox{-}Flat (R)) < \infty$, furthermore $\mathrm{gl.GPD} _{\mathcal{FP}_n\mbox{-}Flat (R)} (R) < \infty$.

(iii) From (ii) we know that $\id (\mathcal{FP}_n\mbox{-}Flat (R)) < \infty$, since $\Proj (R) \subseteq \mathcal{FP}_n\mbox{-}Flat (R)$ then $\id (\Proj (R)) < \infty$. Thus by \cite[Corollary 8.7]{Huerta} and (ii) we get $\mathrm{gl.GID} _{\mathcal{FP}_n\mbox{-}Inj (R)}(R) < \infty$. This implies by the dual of  \cite[Proposition 3.7 (i)]{Becerril22-1} and Remark \ref{Remark1} (2) that $\pd (\mathcal{FP}_n\mbox{-} Inj (R)) < \infty$. 
\end{proof}

\begin{cor}
Every Gorenstein $n$-coherent ring $R$, with $\pd (\Inj (R)) < \infty$ is such that the quantities  
$$\mathrm{l.w.PGFgl}_{ \mathcal{FP}_n\mbox{-}Inj (R^{\op})}(R), \; \; \mathrm{gl.GPD}_{ \mathcal{FP}_n\mbox{-}Flat (R)}(R), \; \; \mathrm{gl.GID}_{ \mathcal{FP}_n\mbox{-}Inj (R)}(R )$$
 are all finite for $n> 1$. 
\end{cor}


 We know that there exists an ascending chain of inclusions 
$$\mathcal{FP}_0 \mbox{-}Inj (R) \subseteq \mathcal{FP}_1\mbox{-}Inj(R) \subseteq \cdots \subseteq \mathcal{FP}_{n} \mbox{-}Inj(R)  \subseteq \cdots \subseteq \mathcal{FP}_{\infty}\mbox{-}Inj (R),$$ 
of classes of $R$-modules \cite[\S 3]{BP}. If $R$ is a fix Gorenstein $n$-coherent ring (in particular if the dimension  $\fd (\mathcal{FP}_{n} \mbox{-}Inj(R) )$ is finite), then for all $k \leq n$ we have that the dimension $\fd (\mathcal{FP}_{k} \mbox{-}Inj(R)  )$ is finite, since $\fd (\mathcal{FP}_{k} \mbox{-}Inj(R)  ) \leq \fd (\mathcal{FP}_{n} \mbox{-}Inj(R)  ) $ for all $k \in \{0,1, \dots , n\}$. With this we have the following.

\begin{cor}
Let  $R$ be a Gorenstein n-coherent ring.  Then for all $k \in \{ 0,1, \dots , n \}$ the left  Gorenstein weak global dimension relative to the class  $\GF _{ \mathcal{FP}_k\mbox{-}Inj (R^{\op})}$  is finite and coincides with $\fd (\mathcal{FP}_k\mbox{-}Inj (R^{\op}))$.
\end{cor}

Recently has been proved by L. W. Christensen, S. Estrada and P. Thompson \cite{Chris-Estrada} without conditions over the ring $R$,  that the classical Gorenstein weak global dimension is symmetric, i.e. the quantities $\mathrm{l.w.Ggl} _{ \Inj (R^{\op}) } (R) $ and $\mathrm{r.w.Ggl} _{ \Inj (R) } (R^{\op})$ always agree.   Actually we have two classes of $\mathcal{FP}_n$-injective modules, the one in $\Modu (R)$ and the other in $\Modu (R ^{\op})$, denoted $ \mathcal{FP}_{n} \mbox{-}Inj(R) $ and $ \mathcal{FP}_{n} \mbox{-}Inj(R^{\op}) $, resp. By definition we know  that $\fd ( \mathcal{FP}_{n} \mbox{-}Inj(R) ) $ and $\fd ( \mathcal{FP}_{n} \mbox{-}Inj(R^{\op}) )$ are both finite when $R$ is a Gorenstein $n$-coherent ring. The author is not sure if these quantities match. For now we only can only assume it. 

\begin{cor}
Let $R$ be a Gorenstein $n$-coherent ring, if the dimensions $\fd ( \mathcal{FP}_{n} \mbox{-}Inj(R) ) $ and $\fd ( \mathcal{FP}_{n} \mbox{-}Inj(R^{\op}) )$ match, then  Gorenstein weak  global dimensions relative to the classes $ \GF _{\mathcal{FP}_{n} \mbox{-}Inj(R ^{\op})}$ and $ \GF _{\mathcal{FP}_{n} \mbox{-}Inj(R )}$ coincide. This is, the following equality is given.

$$ \mathrm{l.w.Ggl}_{\mathcal{FP}_n\mbox{-}Inj (R^{\op})}(R) = \mathrm{r.w.Ggl}_{ \mathcal{FP}_n\mbox{-}Inj (R)}(R ^{\op}).$$
\end{cor} 

As we have seen in Proposition \ref{Interaccion}, there is a close interaction between the classes of relative Gorenstein $R$-modules   coming from the class $\mathcal{FP}_n$.  This behavior has been noted by J. \v{S}aroch and J. \v{S}t'ov\'ich\v{e}k \cite{Stovi}, where by making use of such an interaction it is further proved that every ring is \textit{GF-closed}.    The mentioned interaction has been of interest to other authors. S. Estrada, A. Iacob and M. A. P\'erez  went further in this direction \cite{Estrada20}, showing in particular that  \cite[Theorem 4.11]{Stovi} has a wider scope as is shown  \cite[Theorem 4.2]{Estrada20}, we set out one version of the latter result  as follows.

\begin{pro} \label{SequencesStovi-0}
Let $M $ be a left $R$-module and $1 < n \leq \infty$. The following statements are equivalent. 
\begin{itemize}
\item[(i)] $M \in \GF _{\mathcal{FP}_n\mbox{-}Inj (R^{\op})}$.
\item[(ii)] There is an exact sequence 
$$0 \to K \to L \to M \to 0$$  
 with $K \in \Flat (R)$ and $L \in \Proj \GF _{\mathcal{FP}_n\mbox{-}Inj (R^{\op})} = \GP _{\mathcal{FP}_n\mbox{-}Flat (R)}$, which is $\Hom _{R} (-, \C (R))$-acyclic \footnote{$C$ is a cotorsion $R$-module if $\Ext ^1 _R (\Flat(R), C) =0$, equivalently $\Ext ^{i}_R (\Flat (R), C) =0$ for all $i>0$. Here $\mathcal{C}(R)$ denotes the collection of all these.}.
\item[(iii)] $\Ext ^1 _R (M, C) =0$ for all $C \in \C (R) \cap (\Proj \GF _{\mathcal{FP}_n\mbox{-}Inj (R^{\op})}) ^{\ortogonal}$.
\item[(iv)] There is an exact sequence 
$$0 \to M \to F \to N \to 0,$$
 with $F \in \Flat$ and $N \in \Proj \GF _{\mathcal{FP}_n\mbox{-}Inj (R^{\op})} = \GP _{\mathcal{FP}_n\mbox{-}Flat (R)}$. In particular, we have the equality 
 $$\GF _{\mathcal{FP}_n\mbox{-}Inj (R^{\op})} (R) \cap (\GP _{\mathcal{FP}_n\mbox{-}Flat (R)}) ^{\ortogonal} = \Flat (R).$$ 
\end{itemize}
\end{pro}
\begin{proof}
It follows directly from   \cite[Theorem 2.14]{Estrada20}, since the class $\mathcal{FP}_n\mbox{-}Inj (R^{\op})$ is definable for all $\infty > n > 1$ (so semi-definable also) \cite[Example 2.21 (3)]{Estrada20} and for $n = \infty$ the class $\mathcal{FP}_{\infty}\mbox{-}Inj (R)$ is semi-definable \cite[Example 2.21 (3)]{Estrada20}. While the equality $\Proj \GF _{\mathcal{FP}_n\mbox{-}Inj (R^{\op})} = \GP _{\mathcal{FP}_n\mbox{-}Flat (R)}$ is true for all $n >1 $ by \cite[Theorem 3.6]{Alina20}. 
\end{proof}

We will now see that the above result can be stated for $R$-modules with finite  Gorenstein $\mathcal{FP}_n$-flat dimension. Although we only state part of the equivalence, the rest of the theorem can be taken to such generality. 

\begin{teo} \label{SequencesStovi}
Let $m$ be a non-negative integer and $1< n \leq \infty$. The following statements are equivalent.
\begin{itemize}
\item[(i)] $\Gfd _{\mathcal{FP}_n\mbox{-}Inj (R^{\op})} (M) \leq m$,
\item[(ii)] There is an exact sequence 
$$0 \to M \to F \to N \to 0,$$
 with $\fd (F) \leq m$ and $N \in \GP _{\mathcal{FP}_n\mbox{-}Flat (R)}$.
\end{itemize}
 \end{teo}
 
 \begin{proof}
(i) $\Rightarrow$ (ii). We will proceed by induction on $m$. By Proposition \ref{SequencesStovi-0} the case $m = 0 $ is true. Let $m >0 $. By Lemma \ref{Estable}  we can consider an exact sequence $0 \to K \to H \to M \to 0$ with $H \in \Flat (R)$ and $\Gfd _{\mathcal{FP}_n\mbox{-}Inj (R^{\op})} (K) \leq m-1$, then by induction there is an exact sequence $0 \to K \to H' \to G \to 0$ such that $\fd (H' ) \leq m-1$ and $G \in \GP _{\mathcal{FP}_n\mbox{-}Flat (R)}$. We can construct the following p.o. diagram 
 
 $$\xymatrix{ 
  K_{ }  \ar@{^{(}->}[d]  \ar@{^{(}->}[r] & H _{} \ar@{^{(}->}[d]  \ar@{>>}[r] & M   \ar@{=}[d]  \\
  H'  \ar@{>>}[d] \ar@{^{(}->}[r] & H'' \ar@{>>}[r]\ar@{>>}[d] & M   \\
 G \ar@{=} [r]& G .& 
}$$
Since $H\in \Flat (R)$ and $G \in \GP _{\mathcal{FP}_n\mbox{-}Flat (R)} = \Proj \GF _{\mathcal{FP}_n\mbox{-}Inj (R^{\op})} \subseteq \GF _{\mathcal{FP}_n\mbox{-}Inj (R^{\op})}$  from the exact sequence $0 \to H \to H'' \to G \to 0$ we have that $H'' \in \GF _{\mathcal{FP}_n\mbox{-} Inj (R ^{\op})}$. By Proposition \ref{SequencesStovi-0} there is an exact sequence $ 0 \to H'' \to L \to N \to 0$ with $L \in \Flat (R)$ and $N \in \GP_{\mathcal{FP}_n\mbox{-} Flat (R)}$. Now we construct the following p.o. diagram 
 $$\xymatrix{ 
  H' { }  \ar@{=}[d]  \ar@{^{(}->}[r] & H''  \ar@{^{(}->}[d]  \ar@{>>}[r] & M _{}  \ar@{^{(}->}[d]  \\
  H'   \ar@{^{(}->}[r] & L \ar@{>>}[r] \ar@{>>}[d] & F \ar@{>>}[d]  \\
 & N \ar@{=} [r] &  N.
}$$
In the exact sequence $0 \to H' \to L \to F \to 0$, we have that $L \in \Flat (R) $ and $\fd (H') \leq m-1$, this implies that $\fd (F) \leq m$.  Therefore  $0 \to M \to F \to N \to 0$ is the desired sequence. 

(ii) $\Rightarrow$ (i). Suppose that there is an exact sequence $0 \to M \to F \to N \to 0$ with $\fd (F) \leq m$ and $N \in \GP _{\mathcal{FP}_n\mbox{-} Flat (R)}$. Since $\mathcal{FP}_n\mbox{-} Inj (R ^{\op})$ is a definable class, then by \cite[Theorem 2.13]{Estrada20} the pair $(\Proj \GF _{\mathcal{FP}_n\mbox{-} Inj (R^{\op})}, \Proj \GF _{\mathcal{FP}_n\mbox{-} Inj (R^{\op})}  ^{\ortogonal })$ is an hereditary and complete cotorsion pair. Furthermore using that $P \GF _{\mathcal{FP}_n\mbox{-} Inj (R^{\op})} = \GP _{\mathcal{FP}_n\mbox{-} Flat (R)}$ we have that for $F$ there is an exact sequence 
$$ \gamma : 0 \to E \to L \to F \to 0$$
 with $L \in \GP _{\mathcal{FP}_n\mbox{-} Flat (R)}$ and $ E \in \GP_{\mathcal{FP}_n\mbox{-} Flat (R)} ^{\ortogonal}$. Since $\Gfd _{\mathcal{FP}_n\mbox{-} Inj (R^{\op})} (F) \leq \fd (F) \leq m$ and $L \in \GP _{\mathcal{FP}_n\mbox{-} Flat (R )} \subseteq \GF _{\mathcal{FP}_n\mbox{-} Inj (R^{\op})}$ we have from  $\gamma $ that  $\Gfd _{\mathcal{FP}_n\mbox{-} Inj (R^{\op})} (E) \leq m-1$. Consider the following p.b. diagram
 $$\xymatrix{ 
  E_{}   \ar@{^{(}->}[d]  \ar@{=}[r] & E_{}  \ar@{^{(}->}[d]   &   \\
  Q  ^{} \ar@{>>}[d]  \ar@{^{(}->}[r] & L \ar@{>>}[r] \ar@{>>}[d] & N \ar@{=}[d]  \\
   M ^{}  \ar@{^{(}->}[r]  & F \ar@{>>} [r] &  N.
}$$
Now from the exact sequence $0 \to Q \to L \to N \to 0$ we know that $N, L \in \GP _{\mathcal{FP}_n\mbox{-} Flat (R)}$. From \cite[Corollary 3.33]{BMS} and \ref{Remark1} the class $\GP _{\mathcal{FP}_n\mbox{-} Flat (R)}$ is closed by kernels of epimorphisms, thus  $Q \in \GP _{\mathcal{FP}_n\mbox{-} Flat (R)}$. Finally from the exact sequence $0 \to E \to Q \to M \to 0$, with $Q \in \GP _{\mathcal{FP}_n\mbox{-} Flat (R)} \subseteq \GF _{\mathcal{FP}_n\mbox{-} Inj (R^{\op})}$ and $\Gfd _{\mathcal{FP}_n\mbox{-} Inj (R^{\op})} (E) \leq m-1$ we conclude that $\Gfd _{\mathcal{FP}_n\mbox{-} Inj (R^{\op})} (M) \leq m$. 
 \end{proof}
 
  We are interested in studying orthogonal classes $\GP _{\mathcal{FP}_n\mbox{-} Flat (R)} ^{\ortogonal}$ and $^{\ortogonal} \GI _{\mathcal{FP}_n\mbox{-} Inj (R)}$. Actually we know that the pair $(^{\ortogonal} \GI _{\mathcal{FP}_n\mbox{-} Inj (R)} ,   \GI _{\mathcal{FP}_n\mbox{-} Inj (R)})$ is a hereditary cotorsion pair \cite[Proposition 24]{Gill22}, and if $n>2$ is also a perfect cotorsion pair. While that if $n> 2$ the pair $(\GP _{\mathcal{FP}_n\mbox{-} Flat(R) }, \GP _{\mathcal{FP}_n\mbox{-} Flat (R)} ^{\ortogonal} )$ is a hereditary and complete cotorsion pair \cite[Theorem 3.8]{Alina20}. In what follows, we will look at conditions where the orthogonal classes  $\GP _{\mathcal{FP}_n\mbox{-} Flat (R)} ^{\ortogonal}$ and $^{\ortogonal} \GI _{\mathcal{FP}_n\mbox{-} Inj (R)}$, are the same. Which will give us a \textit{cotorsion triple} \cite{Chen}, and possibly a more extensive notion of \textit{Virtually Gorenstein Ring} \cite[Definition 3.9]{Iranies}.
 
 \begin{pro} \label{GP-ortogonal}
 For   $1 < n \leq \infty$, the containment $\Flat (R) ^{\gorro}  \subseteq \GP _{\mathcal{FP}_n\mbox{-} Flat (R)} ^{\ortogonal}$   is always true. 
 Furthermore, if $\GF ^{\gorro} _{\mathcal{FP}_n\mbox{-} Inj (R ^{\op})} = \Modu (R)$, then the following equality is true
 $$\GP _{\mathcal{FP}_n\mbox{-} Flat (R)} ^{\ortogonal} = \Flat (R) ^{\gorro}.$$
 \end{pro}
 
 \begin{proof}
 Let $F \in \Flat (R) ^{\gorro}$ and $G \in \GP _{\mathcal{FP}_n\mbox{-} Flat (R)} = \Proj \GF _{\mathcal{FP}_n\mbox{-} Inj (R^{\op})}$. From \cite[Theorem 2.9 (2)]{Estrada20} we know that $\Flat (R) \subseteq \Proj \GF ^{\ortogonal} _{\mathcal{FP}_n\mbox{-} Inj (R^{\op})}$, thus, by shift dimension on $F \in \Flat (R) ^{\gorro}$ we get $\Ext ^{i} _R (G, F ) =0$ for all $i > 0 $. Therefore $\Flat (R) ^{\gorro} \subseteq \GP _{\mathcal{FP}_n\mbox{-} Flat (R)} ^{\ortogonal}$.
 
 For other side assume that $\GF ^{\gorro} _{\mathcal{FP}_n\mbox{-} Inj (R^{\op})} = \Modu (R)$ and let $M \in \GP _{\mathcal{FP}_n\mbox{-} Flat (R)} ^{\ortogonal}$. In particular $\Gfd _{\mathcal{FP}_n\mbox{-} Inj (R^{\op})} (M) < \infty$, this implies from Theorem \ref{SequencesStovi} that there is an exact sequence $ 0 \to M \to F \to N \to 0$ with $F \in \Flat  (R)^{\gorro}$ and $N \in  \GP_{\mathcal{FP}_n\mbox{-} Flat (R)}$ which splits. This implies that $M \in \Flat (R) ^{\gorro}$.
 \end{proof}
 
Recall  for a ring $R$ their \textit{left finitistic flat dimension} as follows 

$$ l.\mathrm{FFD} (R) := \sup \{\fd (M): M \in \Modu (R) \mbox{ such that } M \in \Flat (R) ^{\gorro} \},$$
 we will make use of such notion in the following result.
 
 \begin{pro} \label{Jim-idea}
 For $1 < n \leq \infty$, the containment $\Flat (R) ^{\gorro}  \subseteq {^{\ortogonal}\GI _{\mathcal{FP}_n\mbox{-} Inj (R)} }$   is always true. Furthermore, if $\mathcal{FP}_n\mbox{-} Inj (R) \subseteq \Flat (R) ^{\gorro}  $ and $l.\mathrm{FFD}(R)$ are finite then the following equality is true
 $$^{\ortogonal} \GI _{\mathcal{FP}_n\mbox{-} Inj (R)} = \Flat (R) ^{\gorro}.$$
 \end{pro}
 
 \begin{proof}
 Let be $F \in \Flat (R) ^{\gorro}$ and $G \in \GI _{\mathcal{FP}_n\mbox{-} Inj (R)}$. There exists  exact sequences 
 $$0 \to F_j \to  \cdots \to F_0 \to F \to 0  \mbox{ and } 0 \to G' \to I_{j-1} \to \cdots \to I_0 \to G \to 0$$
 such that  $F_i \in \Flat (R)$ for all $i \in \{0,1,\dots ,j\}$ where $j := \fd (F)$ and where $G' \in \GI _{\mathcal{FP}_n\mbox{-} Inj (R)}$  and $I_{i} \in \Inj (R)$ for all $i \in \{0,1 , \dots , j-1\}$.  From \cite[Corollary 5.9]{Stovi14} we know that every Gorenstein injective $R$-module is cotorsion, in particular $\Ext ^{i} _{R} (\Flat (R) , \GI _{\mathcal{FP}_n\mbox{-} Inj (R)}) =0$ for all $i>0$. By shift dimension we get 
 $$\Ext ^{i} _R (F,G) \cong \Ext ^{i+j} _R (F,G') \cong \Ext ^{i} _R (F_j ,G') =0$$
 this is $\Ext _R ^{i} (\Flat (R) ^{\gorro}, \GI _{\mathcal{FP}_n\mbox{-} Inj (R)}) =0$ for all $i>0$. Therefore $\Flat (R) ^{\gorro} \subseteq {^{\ortogonal} \GI _{\mathcal{FP}_n\mbox{-} Inj (R)}}$. 
 
 Now  let us define  $k :=  l.\mathrm{FFD} (R)$, and note that $\Inj (R)\subseteq \Flat (R) ^{\gorro} = \Flat (R) ^{\gorro} _k  $. By \cite[Theorem 4.1.3]{Gobel} the pair $(\Flat (R) ^{\gorro} , (\Flat (R) ^{\gorro})^{\ortogonal})$ is a complete and hereditary  cotorsion pair such that $\Inj (R) \subseteq \Flat (R) ^{\gorro} \cap  (\Flat (R) ^{\gorro})^{\ortogonal}$. Furthermore, such cotorsion pair is injective. To see this take $M \in \Flat (R) ^{\gorro} \cap  (\Flat (R) ^{\gorro})^{\ortogonal}$, always  there is an exact sequence $0 \to M \to I \to C \to 0$ with $I \in \Inj (R)$. Since $I, M \in \Flat (R) ^{\gorro} $ it follows that $C  \in \Flat (R) ^{\gorro} $, in consequence this sequence splits and thus $M \in \Inj (R)$.  We will use this fact to show that $(\Flat  (R)^{\gorro}) ^{\ortogonal} \subseteq \GI_{\mathcal{FP}_n\mbox{-} Inj (R)}$. 
 
 Take $F \in (\Flat (R) ^{\gorro}) ^{\ortogonal} \subseteq (\mathcal{FP}_n\mbox{-} Inj (R)) ^{\ortogonal}$ by the dual of \cite[Proposition 3.16]{BMS}, we only need to construct a left resolution by injectives which still $\Hom _R (\mathcal{FP}_n\mbox{-} Inj (R),-)$-acyclic. We know that there exists an exact sequence $0 \to F_0 \to J_0 \to F \to 0$, with $J_0 \in \Flat (R) ^{\gorro}$ and $F _0 \in (\Flat (R) ^{\gorro})^{\ortogonal}$ thus $J_0 \in \Flat (R) ^{\gorro} \cap  (\Flat (R) ^{\gorro})^{\ortogonal} = \Inj (R)$. Then we can construct an exact complex by injectives $ \cdots \to J_1 \to J_0 \to F \to 0$  with $\Coker(J_{i+1} \to J_i)\in  (\Flat (R)^{\gorro}) ^{\ortogonal} $. This means that such complex is $\Hom _{R} ((\Flat (R) ^{\gorro})  , -)$-acyclic, in particular is $\Hom _{R} (\mathcal{FP}_n\mbox{-} Inj (R),-)$-acyclic.
 
 Now the containment $(\Flat  (R)^{\gorro}) ^{\ortogonal} \subseteq \GI_{\mathcal{FP}_n\mbox{-}  Inj (R)}$ and the fact that $(\Flat (R) ^{\gorro} , (\Flat (R) ^{\gorro})^{\ortogonal})$ is an hereditary cotorsion pair give us that $$ ^{\ortogonal}  ( \GI_{\mathcal{FP}_n\mbox{-} Inj (R)}) \subseteq {^{\ortogonal} ((\Flat (R)^{\gorro}) }^{\ortogonal}) = \Flat (R) ^{\gorro} ,$$
 which ends the proof.
  \end{proof}
 The following result summarises what we have developed so far and will allow us to address the main concept of the following section.
\begin{pro} \label{Ultimo}
Let $R$ be a ring and  fix $2< n \leq \infty$. Assume that   $\mathrm{l.w.Ggl}_{\mathcal{FP}_n\mbox{-}Inj (R^{\op})}(R)$   the Gorenstein weak global dimension respect  the class  $\GF _{\mathcal{FP}_n\mbox{-}Inj (R^{\op})}$ is finite and $\mathcal{FP}_n\mbox{-} Inj (R) \subseteq \Flat (R) ^{\gorro}$. Then the following statements are true.
\begin{itemize}
\item[(i)] The dimension $\mathrm{r.w.Ggl}_{\mathcal{FP}_n\mbox{-}Inj (R)}(R^{\op})$ is finite and coincide with $\mathrm{l.w.Ggl}_{\mathcal{FP}_n\mbox{-}Inj (R^{\op})}(R)$.
\item[(ii)] There exists an hereditary and complete cotorsion triple in $\Modu (R)$
$$(\GP _{\mathcal{FP}_n\mbox{-} Flat (R)}  , \Flat (R) ^{\gorro}  , \GI _{\mathcal{FP}_n\mbox{-} Inj (R)}  ).$$ 
\item[(iii)] The functor $\Hom _R (-,-)$ is right balanced  by $\GP _{\mathcal{FP}_n\mbox{-} Flat (R)}  \times  \GI _{\mathcal{FP}_n\mbox{-} Inj (R)} $ over the whole category $\Modu (R)$.
\item[(iv)] If $\pd (\Inj (R)) < \infty$, then there exists a triangle equivalence 
$$\mathbf{K} (\GP_{\mathcal{FP}_n\mbox{-} Flat (R)} ) \cong \mathbf{K} (\GI _{\mathcal{FP}_n\mbox{-} Inj (R)} ).$$
\end{itemize}
\end{pro}

\begin{proof}
Assume that  $\mathrm{l.w.Ggl}_{\mathcal{FP}_n\mbox{-}Inj (R^{\op})}(R)$ is finite.  Since $\Inj (R^{\op}) \subseteq \mathcal{FP}_n\mbox{-}Inj (R^{\op})$ we have that $\GF _{\mathcal{FP}_n\mbox{-}Inj (R^{\op})} \subseteq \GF _{\Inj (R^{\op})}$, in consequence their  Gorenstein weak global dimensions are such that $\mathrm{l.w.Ggl}_{Inj (R^{\op})}(R) \leq \mathrm{l.w.Ggl}_{\mathcal{FP}_n\mbox{-}Inj (R^{\op})}(R)$. This implies by \cite[Theorem 2.4]{Chris-Estrada} and \cite[Theorem 5.3 (iii)]{Emmanouil} that $m:= l.\mathrm{FFD} (R)$ is finite. Then, we obtain from the containments  $\mathcal{FP}_n\mbox{-} Inj (R) \subseteq \Flat (R) ^{\gorro} \subseteq \Flat (R) ^{\gorro} _m$, that  $\fd (\mathcal{FP}_n\mbox{-} Inj (R)) \leq m$. By \cite[Theorem 5.3 (ii)]{Emmanouil} also know that $\fd (\Inj (R ^{\op}))$ is finite. Therefore from the dual of Proposition \ref{Ikenaga} we have that $\mathrm{r.w.Ggl}_{\mathcal{FP}_n\mbox{-}Inj (R)}(R^{\op})$ is finite. Note that we have the conditions of the Proposition \ref{GP-ortogonal} and \ref{Jim-idea}, thus we also obtain (ii). The assertion (iii) follows directly from   \cite[Proposition 2.6]{Chen}. Now assume that $\pd (\Inj (R)) < \infty$, then the result follows from  Proposition \ref{Globales-iguales} and \cite[Theorem 4.1]{Chen}.

\end{proof}

\begin{rk}\label{RK01} Note that every Gorenstein $n$-coherent ring $R$, satisfies the conditions of Proposition  \ref{Ultimo}.  Furthermore (ii) and (iii) hold over $\Modu (R) $ and $\Modu (R^{\op})$.
\end{rk}

\section{$(\Le, \A)$-Virtually Gorenstein Rings}

We recall from Belingliannis and Reiten \cite{BeRe} that  an artin algebra $R$ is called \textit{Virtually Gorenstein} if $\GP(R) ^{\ortogonal _1} = {^{\ortogonal _1} \GI (R) }$, where $\GP (R)$ denotes the classical Gorenstein projective $R$-modules (resp. $\GI (R)$ denotes the classical Gorenstein injective $R$-modules).  Virtually Gorenstein algebras provide a common generalization of Gorenstein algebras and Cohen-Macaulay type (also called algebras of finite representation). It is known that the Gorenstein symmetric conjecture holds for Virtually Gorenstein Algebras \cite[Theorem 11.3]{Bel}.   Zareh-Khoshchehreh, Asgharzadeh and Divaani-Aazar \cite{Iranies} extended the notion of \textit{Virtually} to the setting of commutative noetherian rings of finite Krull dimension. It is known   that every commutative Gorenstein ring with finite Krull dimension is virtually Gorenstein \cite[Example 3.13(i)]{Iranies}.

In this section we extend the above notions of virtually Gorenstein rings  to the setting of relative Gorenstein $R$-modules with respect to a duality pair $(\Le, \A)$.

\begin{defi}\cite[Definition 2.1]{HJ}
A \textbf{duality pair} over a ring $R$ is a pair $(\Le, \A)$, where $\Le$ is a class of $R$-modules and $\A$ is a class of $R^{\op}$-modules, satisfying the following conditions:
\begin{itemize}
\item[(1)] $M \in \Le$ if and only if $M ^{+} \in \A$, where $M ^{+} := \Hom _{\mathbb{Z}} (M, \mathbb{Q/Z})$.
\item[(2)] $\A$ is closed under direct summands and finite direct sums.
\end{itemize}
\end{defi}

By the Lambek's Theorem, the more natural example of a duality pair is when we consider the class $\F (R)$ of all flat $R$-modules and the class $\I (R^{\op})$ of all injective $R^{\op}$-modules.

Recently Gillespie considers in \cite{Gill19} classes of  Gorenstein projective, Gorenstein flat  $R$-modules and Gorenstein injective $R^{\op}$-modules, relative to a duality pair $(\Le, \A)$ as follows. 

Let $M$ be a $R$-module and $N$ be a $R^{\op}$-module.
\begin{itemize}
\item $M$ is \textbf{Gorenstein $(\Le, \A)$-projective} if $M = Z_0 (\mathbf{P})$ for some exact complex of  projective  $R$-modules $\mathbf{P}$ for which $\Hom _{R} (\mathbf{P}, L)$ is acyclic for all $L \in \Le$.
\item $M$ is \textbf{Gorenstein $(\Le, \A)$-flat} if $M = Z_0 (\mathbf{F})$ for some exact complex of flat $R$-modules for which $A \otimes _R \mathbf{F}$ is acyclic for all $A \in \A$.
\item $N$ is \textbf{Gorenstein $(\Le, \A)$-injective} if $N = Z_0 (\mathbf{I})$ for some exact complex of injective $R^{\op}$-modules $\mathbf{
I}$ for which $\Hom _{R} (A,\mathbf{I})$ is acyclic for all $A \in \A$.
\end{itemize}

We denote the previous classes of $R$-modules  by  $\GP _{ \Le}$, $\GF _{ \A}$ respectively, and the  class of $R^{\op}$-modules by $\GI _{\A}$. We know by \cite[Theorems 5.5 and 5.6]{BP} that the pair 
$$( \mathcal{FP}_n\mbox{-}Flat (R),\mathcal{FP}_n\mbox{-}Inj (R^{\op}))$$ 
of  $\mathcal{FP}_n$-flat  $R$-modules and $\mathcal{FP}_n$-injective $R^{\op}$-modules form a duality pair in $\Modu (R)$. Since the pair $$( \mathcal{FP}_n\mbox{-}Flat (R^{\op}),\mathcal{FP}_n\mbox{-}Inj (R))$$ is also a duality pair in $\Modu (R ^{\op})$, then  we can consider the class of left $R$-modules  
$$ \GP _{ \mathcal{FP}_n\mbox{-}Flat (R) } ,\; \GI _{ \mathcal{FP}_n\mbox{-}Inj(R)} , \;\GF _{ \mathcal{FP}_n\mbox{-}Inj (R ^{\op})}.$$

Thus in the following definition we will consider a duality pair $(\Le, \A)$ in $\Modu (R)$ \textit{which can also be defined in $\Modu (R^{\op})$}   by say $(\tilde{\Le}, \tilde{\A})$, as above. 

\begin{defi}
Let $(\Le, \A )$ be a duality pair in $\Modu (R)$  which can also be defined in   $\Modu (R^{\op})$. We say that $R$ is \textbf{$(\Le , \A)$-Virtually Gorenstein ring} if 
$$\GP _{\Le} ^{\ortogonal _1} = {^{\ortogonal _1} \GI _{\tilde{\A}}}.$$
\end{defi}
Thus, when $(\Le, \A) = (\Flat (R), \Inj (R ^{\op}))$, the notion of $(\Le, \A)$-Virtually Gorenstein ring is precisely that given by Belingliannis and Reiten \cite{BeRe}, in the context of an artin algebra $R$. We have the following result, the proof of which follows from  Proposition  \ref{Ultimo} (ii).
\begin{pro} \label{Pro02}
 Let $R$ be a ring and  fix $2< n \leq \infty$. Assume that  the Gorenstein weak global dimension of $R$ respect  the class  $\GF _{\mathcal{FP}_n\mbox{-}Inj (R^{\op})}$ is finite and $\mathcal{FP}_n\mbox{-} Inj (R) \subseteq \Flat (R) ^{\gorro}$. Then $R$ is a $( \mathcal{FP}_n\mbox{-}Flat (R),\mathcal{FP}_n\mbox{-}Inj (R^{\op}))$-Virtually Gorenstein ring. 
\end{pro}

\begin{rk} Note that from Proposition \ref{Pro02} and Remark \ref{RK01}   every Gorenstein $n$-coherent ring $R$, is  particularly a $( \mathcal{FP}_n\mbox{-}Flat (R),\mathcal{FP}_n\mbox{-}Inj (R^{\op}))$-Virtually Gorenstein ring. 
\end{rk}
 The author is not sure whether the results contained in this paper can be reproduced for a  general duality pair $(\Le, \A)$. That is; What are the conditions on the pair  $(\Le, \A)$  (that can be defined in $\Modu (R ^{\op})$ by say $(\tilde{\Le}, \tilde{\A})$)  in addition to 
 
 $$\tilde{\A} \subseteq \Flat (R) ^{\gorro}  \mbox{ and } \mathrm{l.w.Ggl}_{\A}(R) <\infty, $$
  to guarantee that $R$ will be  a $(\Le, \A)$-Virtually Gorenstein ring? 
  
  In view of \cite[Theorem A6]{BGH} and  \cite[Remark 38]{Gill22} the author conjectures that  $(\Le, \A)$ must be symmetric (see \cite[Definition 2.4]{Gill19}) with $\Inj (R^{\op}) \subseteq \A$ and such that $(\Proj \GP _{\A} ,\Proj \GP _{\A} ^{\ortogonal _1} )$ is a complete cotorsion pair.

\bigskip

\textbf{Acknowledgements} The author want to thank to professor Raymundo Bautista, for several helpful discussions on the results of this article.\\

\textbf{Funding} The author was fully supported by a CONAHCyT posdoctoral fellowship CVU 443002, managed by the Universidad Michoacana de San Nicolas de Hidalgo at the Centro de Ciencias Matemáticas, UNAM.

\subsection*{Declarations}\; \\

\textbf{Ethical Approval}  Not applicable.\\

\textbf{Availability of data and materials} Not applicable.\\

\textbf{Competing interests} The authors declare no competing interests.


\end{document}